\documentclass{article}[12pt]
\usepackage{amsmath}
\usepackage{amssymb}
\usepackage{extarrows}
\usepackage{hyperref}
\newtheorem{theorem}{Theorem}

\newtheorem{lemma}[theorem]{Lemma}

\newtheorem{proposition}[theorem]{Proposition}
\newtheorem{remark}[theorem]{Remark}

\newenvironment{proof}[1][Proof]{\noindent\textbf{#1.} }{\ \rule{0.5em}{0.5em}}
\usepackage{amssymb}
\usepackage{graphicx}
\usepackage[dvipsnames,usenames]{color}

\input{epsf.tex}
\usepackage[affil-it]{authblk}
\author{Ioannis Dimitriou \footnote{idimit@uoi.gr}\footnote{Corresponding author.}}
\affil{\small Department of Mathematics, 
	University of Ioannina, 
	45110, Ioannina, Greece.}
\begin{document}
\title{On the overlap times in queues with dependence under a Farlie-Gumbel-Morgenstern copula}

\maketitle
\begin{abstract}
  In this work, we analyze the steady-state maximum overlap time distribution in a single-server queue by introducing a dependence structure between service and interarrival times under the Farlie-Gumber-Morgenstern copula. We provide explicit expressions by indicating the effect of dependence. We also focus on the steady-state distribution of the minimum overlap time of a customer with its two adjacent customers. A more general dependence structure is also investigated. A numerical example illustrates the effect of dependence on the maximum/minimum overlap times.
    \end{abstract}
    \vspace{2mm}
	
	\noindent
	\textbf{Keywords}: {Queues, maximum overlap time, minimum overlap time, Farlie-Gumbel-Morgenstern copula}

\section{Introduction}
\label{sec10}
Quite recently, the concept of overlap time has become an important metric to describe the complex dynamics of how customers interact in service systems. Since its introduction in the context of COVID-19 and infectious
disease spread \cite{palomo21,kang}, several other works are considered to
explore overlap times in more complicated queueing models and has inspired a large
number of recent papers on the topic, e.g., \cite{boxpen,boxma1,boxma2,boxma3,palomo}. Our focus is on deriving the LST of the steady-state distribution of the maximum and the minimum overlap time in singe-server queues, in which the interarrival and the service times are dependent. The maximum overlap time, stands for the longest amount of
time a customer spends with any other customer in the system, and it quantifies the longest duration during which an individual interacts with another customer within a service system. The minimum adjacent overlap time measures the least amount of time
a customer is going to spend with its two adjacent customers in a service system. Both are important for understanding the spread of infectious diseases. 

The paper in \cite{palomo} studied for the first time the maximum overlap time for a $M/M/1$ queue, while quite recently, the authors in \cite{boxpen} generalized the work in \cite{palomo} to the $G/G/1$ case. All the above mentioned works adopt the assumption of independence among interarrival and service times. It is readily seen that the relation among these two important elements of a queueing system heavily affects the overlap times. In this work, we attempt for the first time to investigate this relation by assuming that the interarrival and service times are no longer independent, but instead, they are dependent based on a class of copulas, namely the Farlie-Gumbel-Morgerstern (FGM) copula. We choose this class of copulas due to its special characteristics, which allow for explicit results.

\subsection{Our contribution}
In this paper, we make the following contributions to the literature:
\begin{itemize}
    \item We derive for the first time the LST of the distribution of the maximum and the minimum overlap time in queues under dependence structures. We consider the $M/G/1$ and the $E(n,\lambda)/G/1$ queues ($E(n,\lambda)$ denotes the $Erlang(n,\lambda)$ interarrival distribution). Contrary to the existing works that assume independence among interarrival and service times, e.g., \cite{boxpen,boxma1,boxma2,boxma3,palomo}, we consider the case where interarrival and service times are dependent based on the FGM copula. To our best knowledge, this is also the first time that the LST of the waiting time distribution, an important result in investigating overlap times, is derived for these models under the dependence based on the FGM copula. In queueing literature, there are very few results that use copulas to measure the dependence between random variables, and thus, this work tries also to fill this gap.
    \item We also considered a more general dependence framework where interarrival times are randomly proportionally dependent to service times, and on top of that the FGM copula plays an important role. 
    \item Numerical results reveal the effect of the dependence on the maximum and minimum overlap time.
\end{itemize}

Let us first briefly introduce the concept of copulas and especially FGM. A bivariate copula $C$ is a joint distribution function on $[0,1]\times[0,1]$ with standard uniform marginal  distributions. Under Sklar's theorem \cite{nelsen}, any bivariate distribution function $F$ with marginals $F_1$ and $F_2$ can be written as $F(x,y)=C(F_{1}(x),F_{2}(y))$, for some copula $C$. For more details on copulas and their properties; see \cite{joe,nelsen}.

Modeling the dependence structure between random variables (r.v.) using copulas has become popular in actuarial science and financial risk management; e.g. \cite{denoi,albteu,cossette1} (non-exhaustive list). In queueing literature there has been some recent works where copulas are used to introduce dependency. In \cite{trapa}, the authors applied a new constructed bound copula to analyze the dependency between two
parallel service times. In \cite{lei}, the authors developed a new method based on copulas to model correlated inputs in discrete-event stochastic systems. Motivated by recent empirical evidence, the authors in \cite{wu1,wu2} used a fluid model with bivariate dependence orders and copulas to study the dependence among service times and patience times in large-scale service systems; see also \cite{gu,wang,rai}. 

In our work we focus on a dependence structure based on the FGM copula, which is defined for $\theta\in[-1,1]$ by
\begin{displaymath}
    C_{\theta}^{FGM}(u_{1},u_{2})=u_{1}u_{2}+\theta u_{1}u_{2}(1-u_{1})(1-u_{2}),\,(u_{1},u_{2})\in[0,1]\times[0,1].
\end{displaymath}
The density function associated to the above expression is
\begin{displaymath}
    c_{\theta}^{FGM}(u_{1},u_{2}):=\frac{\partial^{2}}{\partial u_{1}\partial u_{2}}C_{\theta}^{FGM}(u_{1},u_{2})=1+\theta(1-2u_{1})(1-2u_{2}).
\end{displaymath}

The FGM copula allows for negative and positive dependence and includes the independence copula ($\theta=0$). Our primary motivation for choosing FGM copula is due to its simplicity and its tractability due to its polynomial structure. It has a simple analytical form, which is easy to apply, thus, make it attractive when we are dealing with the modeling of bivariate dependent data. Its simple analytic shape enables closed-form solutions to many problems in applied probability. Moreover, it is a first order approximation of the Plackett copula and of the Frank copula \cite[p. 100, and p. 106, respectively]{nelsen}. For applications of FGM copula in risk theory, health insurance plans, financial risk management, stochastic frontiers and a stereological context, see \cite{cossette1}, \cite{mai}, \cite{barges}, \cite{kim} and the references therein. Often, it is natural to describe the FGM copula as a perturbation of the product copula; e.g. \cite{durante}, inducing moderate dependence between marginals. The class of FGM copulas is a popular choice when working in two dimensions due to its simple shape and the exact calculus of polynomial functions.

In general, the bivariate distribution, say $F_{X,Y}$, of the bivariate random vector $(X,Y)$ with continuous marginals, say $F_{X}$, $F_{Y}$, and with a dependence structure based on the FGM copula is defined for $x,y\in\mathbb{R}^{+}$, $\theta\in[-1,1]$ by
\begin{displaymath}
\begin{array}{rl}
   F_{X,Y}(x,y)=&C_{\theta}^{FGM}(F_{X}(x),F_{Y}(y))\vspace{2mm}\\
     =& F_{X}(x)F_{Y}(y)+\theta F_{X}(x)F_{Y}(y)\left(1-F_{X}(x)\right)\left(1-F_{Y}(y)\right).
\end{array}
\end{displaymath}
 With this, the bivariate density of $(X,Y)$ is given by,
\begin{displaymath}
\begin{array}{rl}
   f_{X,Y}(x,y)=&c_{\theta}^{FGM}(F_{X}(x),F_{Y}(y))f_{X}(x)f_{Y}(y)\vspace{2mm}\\
     =& f_{X}(x)f_{Y}(y)+\theta g(x)f_{Y}(y)\left(2\bar{F}_{Y}(y)-1\right),
\end{array}
\end{displaymath}
where $g(x)=f_{X}(x)(1-2F_{X}(x))$ with Laplace transform $g^{*}(s)=\int_{0}^{\infty}e^{-sx}g(x)dx$, and $\bar{F}_{Y}(y)=1-F_{Y}(y)$.

The rest of the paper is organised as follows. In Section \ref{sec1} we provide the LST of the distribution of the maximum overlap time for the $M/G/1$ and $E(n,\lambda)/G/1$ in the presence of dependence based on the FGM copula. The minimum overlap time for those models is investigated in Section \ref{sec2}, while in Section \ref{sec3} we investigate the case of a more general dependence framework under which interarrival times are randomly proportional to service times, and on top of that, FGM copula plays an important role. Finally, in Section \ref{sec4}, we numerically illustrate our theoretical findings by focusing on the effect of dependence on the mean maximum and minimum overlap time.
\section{The maximum overlap time}\label{sec1}
In this section, we derive our main results for the maximum overlap time for a single-server queue in the presence of dependence between interarrival and service times based on the FGM copula. In doing that, we follow the lines in \cite{boxpen}. For completeness, we state the following lemma.
\begin{lemma} (Lemma 2.2 in \cite{boxpen})
The maximum overlap time has the following decomposition into two
independent terms, viz., the waiting time and max-plus difference of a service time
and inter-arrival time:
\begin{equation}
    M_n = W_n + [S_n - A_n]^{+}. \label{bas}
\end{equation}
\end{lemma}

We first consider the simplest case that corresponds to the $M/G/1$ queue, i.e., we assume that $A\sim exp(\lambda)$. Then,
\begin{equation}
         f_{S,A}(y,x)=f_{S}(y)\lambda e^{-\lambda x}+\theta g(y)\left[2\lambda e^{-2\lambda x}-\lambda e^{-\lambda x}\right].
  \label{biv1}   
\end{equation}
The following theorem provides the LST of distribution of the maximum overlap time for the $M/G/1$ queue under a dependence structure based on FGM copula.
\begin{theorem}\label{th1}
    The LST of the steady-state maximum overlap time, say $m^{+}(s):=E(e^{-sM_{\infty}})$, for the M/G/1 queue with dependence structure under the FGM copula is equal to
    \begin{equation}
        \begin{array}{rl}
             m^{+}(s)=&\left( \frac{s\left[ (2\lambda-s)(\theta g^{*}(\lambda)-\phi_{S}(\lambda))\omega^{*}(\lambda)-\theta(\lambda-s)g^{*}(2\lambda)\omega^{*}(2\lambda)\right]}{(2\lambda-s)(\lambda-s)-\lambda(2\lambda-s)\phi_{S}(s)+\lambda\theta sg^{*}(s)}\right )  \vspace{2mm}\\
             & \times \left( \frac{\lambda}{\lambda-s}\phi_{S}(s)-\frac{s}{\lambda-s}\phi_{S}(\lambda)+s\theta\left[ \frac{g^{*}(\lambda)}{\lambda-s}- \frac{g^{*}(2\lambda)}{2\lambda-s} - \frac{\lambda}{(\lambda-s)(2\lambda-s)}g^{*}(s)\right] \right),
        \end{array}
    \end{equation}
    where
    \begin{equation}
        \begin{array}{rl}
            \omega^{*}(\lambda)= &\frac{2(1-\rho)(\lambda-\tau_{1})}{\tau_{1}(\theta g^{*}(\lambda)-\phi_{S}(\lambda))},\vspace{2mm}  \\
             \omega^{*}(2\lambda)= &\frac{2(1-\rho)(2\lambda-\tau_{1})}{\tau_{1}\theta g^{*}(2\lambda)},
        \end{array}\label{cvq}
    \end{equation}
    and $\tau_{1}$ is the only zero of the equation
    \begin{equation}
        1=\frac{\lambda}{\lambda-s}\phi_{S}(s)-\theta\frac{s\lambda}{(\lambda-s)(2\lambda-s)}g^{*}(s),\label{rou}
    \end{equation}
    such that $Re(\tau_{1})>0$.    Moreover,
    \begin{equation}
        E(M_{\infty})=\frac{\lambda E(S^{2})}{2(1-\rho)}+E(S)-\frac{1-\phi_{S}(\lambda)}{\lambda}+\theta G,\label{exp}
    \end{equation}
\end{theorem}
where for $\rho=\lambda E(S)<1$,
\begin{displaymath}
    \begin{array}{rl}
         G=&\frac{g^{*}(2\lambda)\omega^{*}(2\lambda)(1-\rho-\lambda\theta g^{\prime}(0)+\lambda^{2}E(S^{2}))+2\lambda\omega(\lambda)(g^{*}(\lambda)(\theta g^{*\prime}(0)+\lambda E(S^{2}))-\theta g^{*\prime}(0)\phi_{S}(\lambda))}{4\lambda(1-\rho)^{2}}+\frac{g^{*}(2\lambda)-2g^{*}(\lambda)}{2\lambda}. 
    \end{array}
\end{displaymath}
\begin{proof}
By exploiting the independence of $W_{n}$ and $[S_{n}-A_{n}]^{+}$ in \eqref{bas},
\begin{equation}
    E(e^{-s M_{\infty}})=E(e^{-sW_{\infty}})E(e^{-s[S-A]^{+}})=w^{*}(s)E(e^{-s[S-A]^{+}}).\label{bas1}
\end{equation}
Then,
\begin{equation}
    \begin{array}{rl}
        E(e^{-s[S-A]^{+}}) =&\int_{y=0}^{\infty}\int_{x=0}^{y}e^{-s(y-x)}f_{S,A}(y,x)dxdy +\int_{y=0}^{\infty}\int_{x=y}^{\infty}f_{S,A}(y,x)dxdy \vspace{2mm}\\
         =&\frac{\lambda}{\lambda-s}\phi_{S}(s)-\frac{s}{\lambda-s}\phi_{S}(\lambda)+s\theta\left[ \frac{g^{*}(\lambda)}{\lambda-s}-  \frac{\lambda}{(2\lambda-s)(\lambda-s)}g^{*}(s)-\frac{ g^{*}(2\lambda)}{2\lambda-s}\right].
    \end{array}  \label{p1}
\end{equation}
Note that for $\theta=0$, i.e., when $S,A$ are independent, \eqref{p1} coincides with \cite[eq. (9)]{boxpen}. 

We now focus on the derivation of the LST of the waiting time distribution for the $M/G/1$ queue with dependence among $S,A$ based on the FGM copula. To our best knowledge, this result is not available in the literature. In this work, we denote by $X_{\infty}$, or simply $X$ the steady-state counterpart of $X_{n}$. Then,
\begin{displaymath}
    \begin{array}{rl}
         w^{*}(s)=&E(e^{-s[W+S-A]^{+}})\vspace{2mm}  \\
         =&E\left(\int_{y=0}^{\infty}\int_{x=0}^{W+y}e^{-s(W+y-x)}f_{S,A}(y,x)dxdy\right.\\
         &+\left. \int_{y=0}^{\infty}\int_{x=W+y}^{\infty}f_{S,A}(y,x)dxdy\right)\vspace{2mm}\\
         =&E\left(e^{-sW}\left\{ \int_{y=0}^{\infty}e^{-sy}f_{S}(y)\int_{x=0}^{W+y}\lambda e^{-x(\lambda-s)}dxdy\right.\right.\vspace{2mm}\\&\left.\left.\vspace{2mm}+\theta \int_{y=0}^{\infty}e^{-sy}g(y)\int_{x=0}^{W+y}(2\lambda e^{-x(2\lambda-s)}-\lambda e^{-x(\lambda-s)})dxdy\right\}\right.\vspace{2mm}\\
         &+\left. \int_{y=0}^{\infty}f_{S}(y)\int_{x=W+y}^{\infty}\lambda e^{-x\lambda}dxdy+\theta \int_{y=0}^{\infty}g(y)\int_{x=W+y}^{\infty}(2\lambda e^{-x2\lambda}-\lambda e^{-x\lambda})dxdy\right).
    \end{array}
\end{displaymath}
Tedious but simple calculations lead to
\begin{displaymath}
\begin{array}{r}
    w^{*}(s)[1-\frac{\lambda}{\lambda-s}\phi_{S}(s)+\theta\frac{s\lambda}{(\lambda-s)(2\lambda-s)}g^{*}(s)]=s[\frac{\theta g^{*}(\lambda)-\phi_{S}(\lambda)}{\lambda-s}w^{*}(\lambda)-\frac{\theta g^{*}(2\lambda)}{2\lambda-s}w^{*}(2\lambda)],\end{array}
\end{displaymath}
or equivalently to
\begin{equation}
      w^{*}(s)=\frac{s[(\theta g^{*}(\lambda)-\phi_{S}(\lambda))(2\lambda-s)w^{*}(\lambda)-\theta g^{*}(2\lambda)(\lambda-s)w^{*}(2\lambda)]}{(2\lambda-s)(\lambda-s)-\lambda(2\lambda-s)\phi_{S}(s)+\theta s\lambda g^{*}(s)}.
   \label{v1}
\end{equation}

It is readily seen that in order to obtain $w^{*}(s)$ we need to derive first the unknown terms $w^{*}(\lambda)$, $w^{*}(2\lambda)$ in the right-hand side of \eqref{v1}. This can be done in two steps. First, since $w^{*}(0)=1$, by letting $s\to 0$, simple calculations leads to
\begin{equation}
    2(\theta g^{*}(\lambda)-\phi_{S}(\lambda))w^{*}(\lambda)-\theta g^{*}(2\lambda)w^{*}(2\lambda)=-2(1-\rho),\label{eq1}
\end{equation}
where $\rho:=\lambda E(S)$. To obtain an additional equation we need the following result:
\begin{theorem}\label{rouc}
 Equation \eqref{rou} has exactly one root, say $\tau_{1}$, such that $Re(\tau_{1})>0$, and a second root $\tau_{2}=0$.     
\end{theorem}
\begin{proof}
    The proof is based on a modification of Rouche's theorem \cite{klimenok,adanvan}. Let $z=1-s/k$, and $D_{k}=\{s:|z|=1\}$. In terms of $s$, the contour $D_k$
is a circle of radius $k$ and origin $k$. We let $k\to\infty$, and denote by $D$ the limiting contour. For $s\in D$, excluding $s=0$, or equivalently $z=1$, one can easily show that
\begin{displaymath}
    |\lambda(2\lambda-s)\phi_{S}(s)+\lambda\theta(-s)g^{*}(s)|<|(2\lambda-s)(\lambda-s)|,
\end{displaymath}
since $\frac{\lambda}{\lambda-s}$, $\frac{-s\lambda}{(\lambda-s)(2\lambda-s)}$ are ratios of polynomials with a strictly higher degree at the denominator which leads to
\begin{displaymath}
    \left|\frac{\lambda}{\lambda-s}\phi_{S}(s)+\theta\frac{-s\lambda}{(\lambda-s)(2\lambda-s)}g^{*}(s)\right|\to 0,
\end{displaymath}
on $D$ (excluding $s=0$, or equivalently $z=1$).

Moreover, $\lambda(2\lambda-s)\phi_{S}(s)+\lambda\theta(-s)g^{*}(s)$, and $(2\lambda-s)(\lambda-s)$ are continuous on $D$. Thus, to apply \cite[Theorem 1]{klimenok} we need to show that
\begin{displaymath}
\begin{array}{l}
    \frac{d}{dz}\left[1-\frac{\lambda}{\lambda-k+kz}\phi_{S}(k-kz)-\theta g^{*}(k-kz)\frac{\lambda(kz-k)}{(\lambda-k+kz)(2\lambda-k+kz)}\right]|_{z=1}>0,
    \end{array}
\end{displaymath}
or equivalently, to show that
\begin{displaymath}
    \frac{d}{dz}\left[1-E(e^{-(k-kz)(S-A)})\right]|_{z=1}=-kE(S-A),
\end{displaymath}
where,
\begin{displaymath}
    E(S-A)=\int_{y=0}^{\infty}\int_{x=0}^{\infty}(y-x) 
f_{S,A}(y,x)dxdy=E(S)-\frac{1}{\lambda}<0,\end{displaymath}
due to the stability. Thus, by using  we conclude that inside $D$, the number of roots of \eqref{rou} equals the number of roots of $(2\lambda-s)(\lambda-s)$ inside $D$ minus 1, Thus, we have only one root, say $\tau_{1}$ with $Re(s)>0$. The second one is $\tau_{2}=0$.
\end{proof}

Since $w^{*}(s)$ is an analytic function in $Re(s)>0$, $\tau_{1}$, should be a zero of the numerator in \eqref{v1}. Substituting $s=\tau_{1}$ in the numerator we obtain
\begin{equation}
    (2\lambda-\tau_{1})(\theta g^{*}(\lambda)-\phi_{S}(\lambda))w^{*}(\lambda)-\theta(\lambda-\tau_{1})g^{*}(2\lambda)w^{*}(2\lambda)=0.
    \label{eq2}
\end{equation}
Solving \eqref{eq1}, \eqref{eq2}, we obtain the unknowns $w^{*}(\lambda)$, $w^{*}(2\lambda)$ as given in \eqref{cvq}, thus, $w^{*}(s)$ is fully known in \eqref{v1}. Using \eqref{v1} and \eqref{p1} in \eqref{bas1}, the LST of the maximum overlap time in the $M/G/1$ queue is now completely known. Differentiating with respect to $s$, and letting $s\to 0$, we obtain the expected duration of the maximum overlap time in \eqref{exp}. This completes the proof.
\end{proof}
\subsection{Maximum overlap time for the $E(n,\lambda)/G/1$ under dependence based on the FGM copula}
Let us now consider the case where $A\sim Erlang(n,\lambda)$, $n>1$. Then,
\begin{equation*}
\begin{array}{rl}
         f_{S,A}(y,x)=&f_{S}(y)\frac{\lambda^{n}}{(n-1)!}x^{n-1} e^{-\lambda x}+\theta g(y)\left[2\frac{\lambda^{n}}{(n-1)!}x^{n-1}\sum_{i=0}^{n-1}\frac{(\lambda x)^{i}}{i!} e^{-2\lambda x}-\frac{\lambda^{n}}{(n-1)!}x^{n-1} e^{-\lambda x}\right].\end{array}
  \label{biv2}   
\end{equation*}
The LST of the waiting time is given in the following theorem.
\begin{theorem}\label{th11}
    The LST of the waiting time distribution satisfies the following functional equation
    \begin{equation}
        \begin{array}{l}
             w^{*}(s)\left[1-\left( \frac{\lambda}{\lambda-s} \right)^{n}\phi_{S}(s)-\theta g^{*}(s)\left[2\sum_{i=0}^{n-1}\binom{n+i-1}{i}\left( \frac{\lambda}{2\lambda-s} \right)^{n+i}-\left( \frac{\lambda}{\lambda-s} \right)^{n}\right]\right]\vspace{2mm}  \\
             =\sum_{k=0}^{n-1}\left(\frac{(-\lambda)^{k}}{k!}-\left( \frac{\lambda}{\lambda-s} \right)^{n}\frac{(s-\lambda)^{k}}{k!}\right)\sum_{m=0}^{k}\binom{k}{m}[\phi_{S}^{(m)}(\lambda)-\theta g^{*(m)}(\lambda)]w^{*(k-m)}(\lambda)\vspace{2mm}\\
             +2\theta\sum_{i=0}^{n-1}\binom{n+i-1}{i}\sum_{k=0}^{n+i-1}\left(\frac{1}{2^{n+i}}\frac{(-2\lambda)^{k}}{k!}-\left( \frac{\lambda}{2\lambda-s} \right)^{n+i}\frac{(s-2\lambda)^{k}}{k!}\right)\vspace{2mm}\\
             \times\sum_{m=0}^{k}\binom{k}{m}g^{*(m)}(2\lambda)w^{*(k-m)}(2\lambda),
        \end{array}\label{vn}
    \end{equation}
    where for a function $h(x)$ denote by $h^{(i)}(a)$ its $i$th derivative with respect to $x$ at point $x=a$.
\end{theorem}
\begin{proof}
  We have,
    \begin{displaymath}
    \begin{array}{rl}
         w^{*}(s)=&E(e^{-s[W+S-A]^{+}})\vspace{2mm}  \\
         =&E\left(\int_{y=0}^{\infty}\int_{x=0}^{W+y}e^{-s(W+y-x)}f_{S,A}(y,x)dxdy\right.\\
         &+\left. \int_{y=0}^{\infty}\int_{x=W+y}^{\infty}f_{S,A}(y,x)dxdy\right)\vspace{2mm}\\
         =&E\left(e^{-sW}\left\{ \int_{y=0}^{\infty}e^{-sy}f_{S}(y)\int_{x=0}^{W+y}\frac{\lambda^{n}}{(n-1)!}x^{n-1} e^{-x(\lambda-s)}dxdy\right.\right.\vspace{2mm}\\&\left.\left.+\theta \int_{y=0}^{\infty}e^{-sy}g(y)(2\frac{\lambda^{n}}{(n-1)!}\sum_{i=0}^{n-1}\frac{\lambda^{i}}{i!}\int_{x=0}^{W+y}x^{n+i-1} e^{-x(2\lambda-s)}dx-\frac{\lambda^{n}}{(n-1)!}\int_{x=0}^{W+y}x^{n-1} e^{-x(\lambda-s)}dx)dy\right\}\right.\vspace{2mm}\\
         &+\left. \int_{y=0}^{\infty}f_{S}(y)\int_{x=W+y}^{\infty}\frac{\lambda^{n}}{(n-1)!}x^{n-1} e^{-x\lambda}dxdy\right.\vspace{2mm}\\&\left.+\theta \int_{y=0}^{\infty}g(y)\int_{x=W+y}^{\infty}(2\frac{\lambda^{n}}{(n-1)!}x^{n-1}\sum_{i=0}^{n-1}\frac{(\lambda x)^{i}}{i!} e^{-2\lambda x}-\frac{\lambda^{n}}{(n-1)!}x^{n-1} e^{-\lambda x})dxdy\right).
    \end{array}
\end{displaymath}
Having in mind that $\int_{0}^{x}e^{-ax}x^{k}f(x)dx=(-1)^{k}\mathcal{L}^{(k)}(a)$, where $\mathcal{L}^{(k)}(a)$ denotes the $k$th derivative of the LST of a function $f(x)$ at $x=a$, we can arrive after tedious but straightforward calculations in \eqref{vn}.
\end{proof}
\begin{proposition}\label{prop1}
    For $\theta\neq 0$, the equation
    \begin{equation}
        \left( \frac{\lambda}{\lambda-s} \right)^{n}\phi_{S}(s)+\theta g^{*}(s)\left[2\sum_{i=0}^{n-1}\binom{n+i-1}{i}\left( \frac{\lambda}{2\lambda-s} \right)^{n+i}-\left( \frac{\lambda}{\lambda-s} \right)^{n}\right]=1,
        \label{equ1}
    \end{equation}
    has exactly $3n-2$ roots in the right-half plane with positive real parts and one root equal to zero.
\end{proposition}
\begin{proof}
    The proof is similar to the one in Theorem \ref{rouc} and further details are omitted.
\end{proof}

Note that in order to use \eqref{vn} to derive $w^{*}(s)$, we first need to derive $3n-1$ unknowns, namely $w^{*(k)}(\lambda)$, $k=0,1,\ldots,n-1$, and $w^{*(k)}(2\lambda)$, $k=0,1,\ldots,2n-2$, which arise on the right-hand side of \eqref{vn}. This task is accomplished by using Proposition \ref{prop1}. Note that since $w^{*}(s)$ is an analytic function in the right-half plane, the roots of \eqref{equ1} for $Re(s)>0$, or equivalently the roots of 
\begin{equation}\begin{array}{l}
    (\lambda-s)^{n}(2\lambda-s)^{2n-1}=\lambda^{n}\left[(2\lambda-s)^{2n-1}\phi_{S}(s)+\theta g^{*}(s)\right.\vspace{2mm}\\
    \left.\times\left(2\sum_{i=0}^{n-1}\binom{n+i-1}{i}\lambda^{i}(\lambda-s)^{n}(2\lambda-s)^{n-i-1}-(2\lambda-s)^{2n-1}\right)\right],
    \end{array}
\end{equation}
for $Re(s)>0$ should be roots of the right-hand side of \eqref{vn}. By substituting these roots on the right-hand side of \eqref{vn} we obtain $3n-2$ linear equations. The last one is derived by that fact that $w^{*}(0)=1$. Solving this linear system of $3n-1$ equations we obtain $w^{*(k)}(\lambda)$, $k=0,1,\ldots,n-1$, and $w^{*(k)}(2\lambda)$, $k=0,1,\ldots,2n-2$. Thus, $w^{*}(s)$ is considered known by combining Theorem \ref{th11}, and Proposition \ref{prop1}.

Then, for $A\sim Erlang(n,\lambda)$, $n>1$,
    \begin{equation}
        \begin{array}{rl}
             E(e^{-[S-A]^{+}})=&\int_{y=0}^{\infty}\int_{x=0}^{y}e^{-s(y-x)}f_{S,A}(y,x)dxdy +\int_{y=0}^{\infty}\int_{x=y}^{\infty}f_{S,A}(y,x)dxdy \vspace{2mm}\\=&\left(\frac{\lambda}{\lambda-s}\right)^{n}\phi_{S}(s)+\left(1-\left(\frac{\lambda}{\lambda-s}\right)^{n}\right)\sum_{k=0}^{n-1}\frac{(s-\lambda)^{k}}{k!}\phi_{S}^{(k)}(\lambda) \vspace{2mm} \\
             &+\theta\left\{ g^{*}(s)(2\sum_{i=0}^{n-1}\binom{n+i-1}{i}\left(\frac{\lambda}{2\lambda-s}\right)^{n+i}+\left(\frac{\lambda}{\lambda-s}\right)^{n})\right.\vspace{2mm}\\
             &\left.+2\sum_{i=0}^{n-1}\binom{n+i-1}{i}\sum_{k=0}^{n+i-1}\left(\frac{1}{2^{n+i}}\frac{(-2\lambda)^{k}}{k!}-\left(\frac{\lambda}{2\lambda-s}\right)^{n+i}\frac{(s-2\lambda)^{k}}{k!}\right)g^{*(k)}(2\lambda)\right.\vspace{2mm}\\
             &\left. + \sum_{k=0}^{n-1}\left(\frac{(-\lambda)^{k}}{k!}-\left(\frac{\lambda}{\lambda-s}\right)^{n}\frac{(s-\lambda)^{k}}{k!}\right)g^{*(k)}(\lambda) \right\}.
        \end{array}\label{wa1}
    \end{equation}

The LST of the distribution of the maximum overlap time for the $E(n,\lambda)/G/1$ queue is given by multiplying $w^{*}(s)$ derived in Theorem \ref{th11} and \eqref{wa1}.
\section{The minimum overlap time}\label{sec2}
 We now turn our attention to the derivation of the LST of the distribution of the minimum overlap time of the $n$th customer, say $V_{n}$ with the previous and the next one. Following \cite{boxpen},
 \begin{displaymath}
     V_{n}=min(W_{n},[W_{n}+S_{n}-A_{n}]^{+}).
 \end{displaymath}
 Then, we have the following result:
 \begin{theorem}\label{th2}
     The LST of the steady-state minimum overlap time $V_{\infty}$ for the $M/G/1$ queue under dependence based on the FGM copula is given by:
     \begin{equation}
         \begin{array}{rl}
              E(e^{-s V_{\infty}})=& E(e^{-sW_{\infty}})\left[ 2+\frac{s}{\lambda-s}\phi_{S}(\lambda)-\frac{\lambda}{\lambda-s}\phi_{S}(s)+\theta s(\frac{\lambda}{(2\lambda-s)(\lambda-s)}g^{*}(s)+\frac{g^{*}(2\lambda)}{2\lambda-s}-\frac{g^{*}(\lambda)}{\lambda-s})\right],
         \end{array}\label{minn}
     \end{equation}
     
 \end{theorem}
 \begin{proof}
     Following \cite[Theorem 3.2]{boxpen},
     \begin{equation}
          E(e^{-s V_{\infty}})= E(e^{-sW_{\infty}})[1+P(S>A)-E(e^{-s(S-A)}1(S>A))].\label{vb}
     \end{equation}
     Note that,
     \begin{displaymath}
         \begin{array}{rl}
              P(S>A)=&\int_{y=0}^{\infty}\int_{x=0}^{y}f_{S,A}(y,x)dxdy  \vspace{2mm}\\
              =& \int_{y=0}^{\infty}\int_{x=0}^{y}[f_{S}(y)\lambda e^{-\lambda x}+\theta g(y)\left[2\lambda e^{-2\lambda x}-\lambda e^{-\lambda x}\right]]dydx\vspace{2mm}\\
              =&1-\phi_{S}(\lambda)+\theta(g^{*}(\lambda)-g^{*}(2\lambda)).
         \end{array}
     \end{displaymath}
     Moreover,
     \begin{displaymath}
         \begin{array}{rl}
            E(e^{-s(S-A)}1(S>A))=  & \int_{y=0}^{\infty}\int_{x=0}^{y}e^{-s(y-x)}f_{S,A}(y,x)dxdy \vspace{2mm} \\
              =&\frac{\lambda}{\lambda-s}(\phi_{S}(s)-\phi_{S}(\lambda))\vspace{2mm}\\&+\theta[\frac{\lambda}{\lambda-s}g^{*}(\lambda)-\frac{2\lambda}{2\lambda-s}g^{*}(2\lambda)-\frac{s\lambda}{(\lambda-s)(2\lambda-s)}g^{*}(s)] 
         \end{array}
     \end{displaymath}
     Substituting the last two derivations in \eqref{vb}, we result in \eqref{minn}, where $w^{*}(s)=E(e^{-sW_{\infty}})$ is given in \eqref{v1}.
 \end{proof}
 \begin{remark}
     Note that for $\theta=0$, i.e., the independent case, Theorem \ref{th2} reduces to \cite[Theorem 3.2]{boxpen}.
 \end{remark}

 Let's focus now on the minimum overlap time for the $E(n,\lambda)/G/1$ queue. The next theorem is the main result.
 \begin{theorem}
     \label{th3}
     The LST of the steady-state minimum overlap time $V_{\infty}$ for the $E(n,\lambda)/G/1$ queue under dependence based on the FGM copula is given by:
     \begin{equation}
              E(e^{-s V_{\infty}})=w^{*}(s)r(s), 
        \label{minn1}
     \end{equation}
     where $w^{*}(s)$ as derived in Theorem \ref{th11}, and 
     \begin{displaymath}
         \begin{array}{l}
              r(s)=2- \left( \frac{\lambda}{\lambda-s} \right)^{n}\phi_{S}(s)-\sum_{k=0}^{n-1}[\frac{(-\lambda)^{k}}{k!}-\left( \frac{\lambda}{\lambda-s} \right)^{n}\frac{(s-\lambda)^{k}}{k!}]\phi_{S}^{(k)}(\lambda)\vspace{2mm} \\
               +\theta\left\{\sum_{k=0}^{n-1}[\frac{(-\lambda)^{k}}{k!}-\left( \frac{\lambda}{\lambda-s} \right)^{n}\frac{(s-\lambda)^{k}}{k!}]g^{*(k)}(\lambda)-g^{*}(s)\left[2\sum_{i=0}^{n-1}\binom{n+i-1}{i}\left( \frac{\lambda}{2\lambda-s} \right)^{n+i}-\left( \frac{\lambda}{\lambda-s} \right)^{n}\right]\right.\vspace{2mm}\\
               -\sum_{i=0}^{n-1}\binom{n+i-1}{i}\sum_{k=0}^{n+i-1}[\frac{1}{2^{n+i}}\frac{(-2\lambda)^{k}}{k!}-\left( \frac{\lambda}{2\lambda-s} \right)^{n+i}\frac{(s-2\lambda)^{k}}{k!}]g^{*(k)}(2\lambda).
         \end{array}
     \end{displaymath}
     
 \end{theorem}
 \section{A more general case}\label{sec3}
 Consider now the case where $S_{n}=\Omega_{n}A_{n}+J_{n}$, where $\Omega_{n}$ is a discrete r.v. independent of any other r.v. in the system that takes values in $E=\{a_{1},\ldots,a_{N}\}$ with $a_{i}<1$, $i=1,\ldots,N$. Moreover, assume that $P(\Omega_{n}=a_{i})=p_{i}$, $i=1,\ldots,N$ with $\sum_{i=1}^{N}p_{i}=1$. On top of that, we consider the case where $S_{n}$, $J_{n}$ are also dependent based on FGM, with $J_{n}\sim exp(\lambda)$. Then, the joint p.d.f. of $\{(S_{n},J_{n})\}_{n\in\mathbb{N}}$ is given by \eqref{biv1}. Then, for $\bar{a}_{i}=1-a_{i}$, it is readily seen that
 \begin{equation}
     E(e^{-s[S_{n}-A_{n}]^{+}})=E(e^{-s[(1-\Omega_{n})S_{n}-J_{n}]^{+}})=\sum_{i=1}^{N}p_{i} E(e^{-s[\bar{a}_{i}S_{n}-J_{n}]^{+}}),
 \label{c1}
 \end{equation}
 where
 \begin{displaymath}
 \begin{array}{rl}
     E(e^{-s[\bar{a}_{i}S_{n}-J_{n}]^{+}})=&\int_{y=0}^{\infty}\int_{x=0}^{\bar{a}_{i}y}e^{-s(\bar{a}_{i}y-x)}f_{S,J}(y,x)dxdy +\int_{y=0}^{\infty}\int_{x=\bar{a}_{i}y}^{\infty}f_{S,J}(y,x)dxdy \vspace{2mm}\\
         =&\frac{\lambda}{\lambda-s}\phi_{S}(\bar{a}_{i}s)-\frac{s}{\lambda-s}\phi_{S}(\bar{a}_{i}\lambda)+s\theta\left[ \frac{g^{*}(\bar{a}_{i}\lambda)}{\lambda-s}-\frac{g^{*}(2\bar{a}_{i}\lambda)}{2\lambda-s}-\frac{\lambda}{(2\lambda-s)(\lambda-s)}g^{*}(\bar{a}_{i}s)\right].
     \end{array}
 \end{displaymath}
 We now turn our attention to the derivation of $w^{*}(s)=E(e^{-sW_{\infty}})$ for this model. Note that,
 \begin{displaymath}
    E(e^{-sW_{n+1}})=E(e^{-s[W_{n}+(1-\Omega_{n})S_{n}-J_{n}]^{+}})=\sum_{i=1}^{N}p_{i} E(e^{-s[W_{n}+\bar{a}_{i}S_{n}-J_{n}]^{+}})
 \end{displaymath}
Then, for $i=1,\ldots,N$,
\begin{displaymath}
    \begin{array}{rl}
       E(e^{-s[W+\bar{a}_{i}S-J]^{+}})=&E\left(\int_{y=0}^{\infty}\int_{x=0}^{W+\bar{a}_{i}y}e^{-s(W+\bar{a}_{i}y-x)}f_{S,J}(y,x)dxdy\right.\\
         &+\left. \int_{y=0}^{\infty}\int_{x=W+\bar{a}_{i}y}^{\infty}f_{S,J}(y,x)dxdy\right)\vspace{2mm}\\
         =&E(e^{-sW})[\frac{\lambda}{\lambda-s}\phi_{S}(\bar{a}_{i}s)-\theta\frac{s\lambda}{(\lambda-s)(2\lambda-s)}g^{*}(\bar{a}_{i}s)]\vspace{2mm}\\
   & +s[\frac{\theta g^{*}(\bar{a}_{i}\lambda)-\phi_{S}(\bar{a}_{i}\lambda)}{\lambda-s}w^{*}(\lambda)-\frac{\theta g^{*}(2\bar{a}_{i}\lambda)}{2\lambda-s}w^{*}(2\lambda)].
    \end{array}
\end{displaymath}
Simple calculations imply that
\begin{equation}
    \begin{array}{l}
         w^{*}(s)[1-\sum_{i=1}^{N}p_{i}\left( \frac{\lambda}{\lambda-s}\phi_{S}(\bar{a}_{i}s)-\theta\frac{s\lambda}{(\lambda-s)(2\lambda-s)}g^{*}(\bar{a}_{i}s)\right)] \vspace{2mm}\\
        = s\sum_{i=1}^{N}p_{i}\left( \frac{\theta g^{*}(\bar{a}_{i}\lambda)-\phi_{S}(\bar{a}_{i}\lambda)}{\lambda-s}w^{*}(\lambda)-\frac{\theta g^{*}(2\bar{a}_{i}\lambda)}{2\lambda-s}w^{*}(2\lambda)\right).
    \end{array}\label{fgq}
\end{equation}
Note that as in Theorem \ref{th1}, we have to obtain first $w^{*}(\lambda)$, $w^{*}(2\lambda)$. Following similar arguments as in Theorem \ref{rouc}, we can show that the term in brackets in the left-hand side of \eqref{fgq} has two roots, namely $\tau_{1}$, such that $Re(\tau_{1})>0$, and $\tau_{2}=0$. Thus, we can use the fact that $w^{*}(0)=1$, and that $s=\tau_{1}$ should also vanish the right-hand side of \eqref{fgq} to obtain two linear equations to obtain $w^{*}(\lambda)$, $w^{*}(2\lambda)$. Thus, $w^{*}(s)$ is considered known. 

Now, by multiplying $w^{*}(s)$ derived above with \eqref{c1} we obtain the LST of the distribution of the maximum overlap time for the model where $S_{n}=\Omega_{n}A_{n}+J_{n}$, where $\Omega_{n}$ such that $P(\Omega_{n}=a_{i})=p_{i}$, $i=1,\ldots,N$, with $a_{i}<1$, $i=1,\ldots,N$, and $\{(S_{n},J_{n})\}_{n\in\mathbb{N}}$ are dependent based on FGM copula.

Using similar arguments as in Theorem \ref{th2}, the LST of the steady-state minimum overlap time $V_{\infty}$ for the model at hand equals $E(e^{-s V_{\infty}})=w^{*}(s)r(s)$, where $w^{*}(s)$ as derived in \eqref{fgq} and now
\begin{displaymath}
\begin{array}{rl}
    r(s)=&2+\sum_{i=1}^{N}p_{i}\left[\phi_{S}(\bar{a}_{i}\lambda)\frac{s}{\lambda-s}-\frac{\lambda}{\lambda-s}\phi_{S}(\bar{a}_{i}s)+\theta s(\frac{\lambda}{(2\lambda-s)(\lambda-s)}g^{*}(\bar{a}_{i}s)+\frac{g^{*}(2\bar{a}_{i}\lambda)}{2\lambda-s}-\frac{g^{*}(\bar{a}_{i}\lambda)}{\lambda-s})\right].\end{array}
\end{displaymath}
\section{Numerical results}\label{sec4}
We now illustrate the theoretical findings. In particular, we illustrate the effect of dependence among interarrival and service times based on the FGM copula on the mean maximum and mean minimum overlap time, by considering the simple case of an $M/M/1$ queue under the FGM copula based on the results derived in Theorems \ref{th1}, \ref{th2}.

In Figure \ref{fig1}, we illustrate the effect of parameter $\theta$, i.e., the dependence based on the FGM copula on the mean duration of the maximum and the minimum overlap time, when we assume exponentially distributed interarrival and service times. We observe that when $\theta$ increases, the maximum and the minimum overlap time also increases. This is because when $\theta$ is positive (resp. negative), the probability of having a long service time increases as the time elapsed since the last arrival increases (resp. decreases). Moreover, by increasing $\rho$, the mean maximum and minimum overlap time significantly increases, as expected. Note that in the case of independence of $S$, $A$ (i.e., $\theta=0$, the case in \cite{boxpen}), the mean maximum and minimum overlap time are given at the point where the lines in Figure \ref{fig1} cross the vertical axis. 
\begin{figure}[!ht]
    \centering
    \includegraphics[width=1\linewidth]{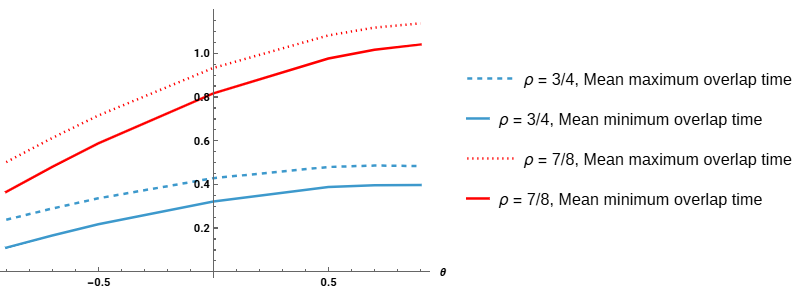}
    \caption{Mean maximum/minimum overlap time versus $\theta$}
    \label{fig1}
\end{figure}
\section*{Acknowledgements} The author gratefully acknowledges the Empirikion Foundation, Athens, Greece (\href{https://www.empirikion.gr/}{www.empirikion.gr}) for the financial support of this work. This work is dedicated to the memory of Antonis Kontis. 
\bibliographystyle{abbrv} 
 \bibliography{mybibitem}

\end{document}